\documentclass[12pt]{article} 

\usepackage{amsmath}
\usepackage{amsthm}
\usepackage{amsfonts}
\usepackage{mathrsfs}
\usepackage{stmaryrd}
\usepackage{setspace}
\usepackage{fullpage}
\usepackage{amssymb}
\usepackage{breqn}
\usepackage{enumitem}
\usepackage{bbold} 
\usepackage{authblk}
\usepackage{comment}
\usepackage{hyperref}
\usepackage{pgf,tikz}
\usepackage{graphicx}
\usetikzlibrary{decorations.pathreplacing,arrows}
\tikzstyle{vertex}=[circle,draw=black,fill=black,inner sep=0,minimum size=3pt,text=white,font=\footnotesize]

\newtheorem{thm}{Theorem}[section]%[chapter]

\newtheorem{lemma}[thm]{Lemma}

\newtheorem{proposition}[thm]{Proposition}

\newtheorem{clm}[thm]{Claim}

\newtheorem*{lemma*}{Lemma}
\newtheorem*{proposition*}{Proposition}
\newtheorem*{theorem*}{Theorem}

\newcommand\ex{\ensuremath{\mathrm{ex}}}
\newcommand\cA{{\mathcal A}}
\newcommand\cB{{\mathcal B}}

\newcommand\cE{{\mathcal E}}
\newcommand\cF{{\mathcal F}}
\newcommand\cG{{\mathcal G}}
\newcommand\cH{{\mathcal H}}

\newcommand\cN{{\mathcal N}}

\newcommand{\ignore}[1]{}

\title{Generalized Tur\'an results for intersecting cliques}
\author{Dániel Gerbner$^1$, Bal\'azs Patk\'os$^{1,2}$\\ \small $^1$ Alfr\'ed R\'enyi Institute of Mathematics\\ \small $^2$ Moscow Institute of Physics and Technology}
\date{}

\begin{document}

\maketitle

\begin{abstract}
    For fixed graphs $F$ and $H$, the generalized Turán problem asks for the maximum number $\ex(n,H,F)$ of copies of $H$ that an $n$-vertex $F$-free graph can have. In this paper, we focus on the case when $F$ is $B_{r,s}$, the graph consisting of two cliques of size $r$ sharing $s$ common vertices. We determine $\ex(n,K_t,B_{r,0})$, $\ex(n,K_t,B_{r,1})$ and $\ex(n,K_{a,b},B_{3,1})$ for all values of $a,b,r,t$ if $n$ is large enough.
\end{abstract}

\section{Introduction}

A central question in extremal graph theory, the so-called Turán problem asks for the maximum number $\ex(n,F)$ of edges that an $n$-vertex graph $G$ can have without containing $F$ as a subgraph. Graphs with this property are called \textit{$F$-free}. The asymptotics of $\ex(n,F)$ is given by the celebrated Erd\H os-Stone-Simonovits theorem \cite{ES} if the chromatic number of $F$ is at least three. For results and open problems in the case when $F$ is bipartite, see the survey by Füredi and Simonovits \cite{FS}.

A natural generalization of this problem is to maximize the number of copies of some other graph $H$ while forbidding $F$ as subgraph. This maximum is denoted by $\ex(n,H,F)$. More precisely, we denote by $\cN(H,G)$ the number of (unlabeled) copies of $H$ in $G$, and $\ex(n,H,F):=\max\{\cN(H,G): G \text{ is an $F$-free graph on $n$ vertices}\}$. So the original problem is the $H=K_2$ case and $\ex(n,F)=\ex(n,K_2,F)$. More generally, for a family $\cF$ of graphs, we denote by $\ex(n,H,\cF)$ the maximum number of copies of $H$ in $n$-vertex graphs that do not contain any member of $\cF$. After some very interesting but sporadic results \cite{BGy,G,Hetal}, these so-called generalized Turán problems were first addressed systematically by Alon and Shikhelman \cite{AS}.

In this paper we study the case where $H$ consists of two cliques sharing some vertices.
Let us denote by $B_{r,s}$ the graph consisting of two $r$-cliques sharing exactly $s$ vertices. We also call it a \textit{generalized book graph}. We call the vertices shared by the two $r$-cliques \textit{rootlet} vertices, and the other vertices of the book graph are \textit{page} vertices. The page vertices are partitioned into two pages, according to which of the two $r$-cliques they belong to.

Let us denote by $G_1+G_2$ the graph consisting of a vertex disjoint pair of copies of $G_1$ and $G_2$ and by $kG$ the graph consisting of $k$ vertex-disjoint copies of $G$. Let $T(m,s)$ denote the Turán graph, which is the complete $s$-partite graph on $m$ vertices with each part having order $\lfloor m/s\rfloor$ or $\lceil m/s\rceil$. For graphs $G_1,G_2$, their join $G_1\vee G_2$ denotes the graph obtained by taking vertex disjoint copies of $G_1,G_2$ and joining every pair $v_1,v_2$ of vertices with $v_1\in V(G_1),v_2\in V(G_2)$. For a set $U\subset V(G)$, we denote by $G[U]$ the subgraph of $G$ induced by $U$, i.e., the subgraph we obtain by deleting the vertices not in $U$.

As it was observed by  Clark, Entringer, McCanna, and Sz\'ekely \cite{cemsz}, the celebrated 6-3 theorem of Ruzsa and Szemer\'edi \cite{rsz} can be reformulated the following way: the largest number of edges in an $n$-vertex graph where every edge is contained in exactly one triangle is $o(n^2)$ but at least $n^{2-o(1)}$. This implies the same bounds on $\ex(n,K_3,B_{3,2})$. Gowers and Janzer \cite{GJ} (motivated by a rainbow variant of generalized Tur\'an problems \cite{gmmp}) generalized this by showing that for $2\le s<r$, we have $n^{s-o(1)}<\ex(n,K_r,\{B_{r,s},B_{r,s+1},\dots,B_{r,r-1}\})=o(n^s)$. Liu and Wang \cite{LW} initiated the study of $\ex(n,K_r,B_{r,s})$. They determined its value exactly for $s=0$ and $s=1$ in the case $n$ is large enough, and gave bounds in the case of other values of $s$.

In this paper we extend these investigations to counting other graphs. Let us first discuss other results that fit into this setting.
If $s=0$, we forbid two vertex-disjoint copies of $K_r$. Moon \cite{moon} showed that $\ex(n,kK_r)=|E(K_{k-1}\vee T(n-k+1,r-1))|$. Concerning generalized Tur\'an problems, $\ex(n,H,kG)$ was studied in \cite{gmv}, in particular the order of magnitude of $\ex(n,K_\ell,k K_r)$ was determined there. 

If $s\ge 2$, then $B_{r,s}$ has a color-critical edge,  i.e., an edge whose deletion decreases the chromatic number. Simonovits \cite{S} showed that for an $r$-chromatic graph $F$ with a color-critical edge, $\ex(n,F)=|E(T(n,r-1)|$. Ma and Qiu \cite{MQ} showed that if $k<r$, then $\ex(n,K_k,F)=\cN(K_k,T(n,r-1))$. Gerbner and Palmer \cite{gerpal} and Gerbner \cite{ger} showed that we also have $\ex(n,P_3,F)=\cN(P_3,T(n,r-1))$. Gerbner \cite{ger2} presented a theorem that determines the exact value of $\ex(n,H,B_{r,s})$ for a class of graphs $H$ if $n$ is large enough. 

In the case $s=1$, we have $\ex(n,B_{r,1})=|E(T(n,r-1)|+1$ and the extremal construction is the Tur\'an graph with an arbitrary additional edge. This was proved in \cite{efgg} for $r=3$ and in \cite{cgpw} for larger $r$.
%We will be specifically interested in $B_{3,1}$, which is also called the 2-fan graph. 
Gerbner and Palmer \cite{gerpal} determined $\ex(n,C_4,B_{3,1})$. $B_{r,1}$ has a color-critical vertex, i.e., a vertex whose deletion decreases the chromatic number (from $r$ to $r-1$). Gerbner \cite{ger2} determined $\ex(n,H,F)$ for every $r$-chromatic graph $F$ with a color-critical vertex if $H$ is a complete balanced $(r-1)$-partite graph $K_{a,\dots,a}$ with $a$ large enough. In particular, this determines $\ex(n,K_{a,a},B_{3,1})$ for every $a$.

For a family of graphs $\cH$, we denote by $\ex(n,\cH,F)$ the largest value of $\sum_{H\in \cH} \cN(H,G)$ $n$-vertex $F$-free graphs $G$. The study of counting multiple graphs in generalized Tur\'an problems was initiated in \cite{gerbn}.
Now we are ready to state our results.

\begin{thm}\label{2kr}
For any $r$ and large enough $n$, we have the following:

(i) if $k< r$, then $\ex(n,K_k,2K_r)=\cN(K_k,K_1\vee T(n-1,r-1))$,

(ii) if $r\le k<2r$, then $\ex(n,\{K_k,K_{k+1},\dots,K_{2r-1}\},2K_r)=\cN(K_k,K_{2k-2r+1}\vee T(n-2k+2r-1,2r-k-1)$.
\end{thm}

Note that the above theorem determines $\ex(n,K_k,2K_r)$ for every pair of $s$ and $r$ if $n$ is large enough. The second statement gives a bit more: if we count the copies of larger cliques in addition to $K_k$, then we obtain the same bound. The extremal graph is the same for $\ex(n,\{K_k,K_{k+1},\dots,K_{2r-1}\},2K_r)$ and $\ex(n,K_k,2K_r)$, and it does not contain cliques of order more than $k$. 

\begin{thm}\label{genbrs}
For any $r\ge 3$, $1\le s\le r-1$, and $1\le t <r-s$, we have \[\ex(n,K_{r+t},B_{r,s})=\Omega(n^{r-s-t-1}).\]

For any $r\ge 3$, $2s+t+1<r$ and $n$ large enough, we have  \[\ex(n,K_{r+t},B_{r,s})=\Theta(n^{r-s-t-1}).\]

For any $t\ge 1$, $t+3<r$ and $n$ large enough, we have \[\ex(n,K_{r+t},B_{r,1})=\cN(K_{r+t},K_{2t+2}\vee T(n-2t-2,r-t-2)).\]
\end{thm}

Observe that the above theorem determines $\ex(n,K_k,B_{r,1})$ for every $k>r$ and $n$ large enough. Consider now the case $k<r$. For $s=0$, the value of $\ex(n,K_k,B_{r,s})$ is determined exactly for sufficiently large $n$ by Theorem \ref{2kr}. We have mentioned a result of Ma and Qiu \cite{MQ} earlier about graphs with a color-critical edge, which determines $\ex(n,K_k,B_{r,s})$ for sufficiently large $n$ if $k<r$ and $s\ge 2$. We now deal with the remaining case $s=1$. Let $T^+(n,r-1)$ denote the graph we obtain from $T(n,r-1)$ by adding an edge to a smallest part.

\begin{thm}\label{newn} If $k<r$ and $n$ is sufficiently large, then
$\ex(n,K_k,B_{r,1})=\cN(K_k,T^+(n,r-1))$.
\end{thm}

%In the case $k=2$, this gives a special case of a theorem of Chen, Gould, Pfender and Wei \cite{} on the ordinary Tur\'an number intersecting cliques. 

In the case $r=3$, we can obtain a much more general result.

\begin{thm}\label{b31}
%If $H$ is a complete bipartite graph, then
For any integers $a\le b$ and $n$ large enough, we have that $\ex(n,K_{a,b},B_{3,1})=\cN(K_{a,b},T)$ for some $n$-vertex graph $T$ that is obtained from a complete bipartite graph by adding an edge.
\end{thm}

For given $a$ and $b$, a straightforward optimization shows what $T$ is. In the case $K_{a,b}$ is not a star, i.e., $a,b\ge 2$, the extra edge of $T$ cannot be in any copy of $K_{a,b}$, thus a complete bipartite graph $K_{m,n-m}$ is also an extremal graph.

\section{Forbidding $2K_r$ and counting cliques}

%\begin{thm}[Frankl \cite{F}]\label{frankl}
%Let $1\le t\le k\le n$ and $n>n_1(k,t)$ and let $\cF\subseteq \binom{[n]}{k}$ be a $t$-intersecting family. DANI: non-trivial?

%(a) If $t+1\le k \le 2t+1$, then $|\cF|\le |\cF_1(n,k,t)|$, 

%where $\cF_1(n,k,t)=\{F\in \binom{[n]}{k}:|[1,t+2]\cap F|\ge t+1\}. $

%(b) If $k>2t+1$, then
%$|\cF|\le |\cF_2(n,k,t)|$, 

%where $\cF_2(n,k,t)=\{F\in \binom{[n]}{k}:[1,t]\subseteq F,V\cap [1+t,k+1]\neq \emptyset\}\cup \{[1,k+1]\setminus \{i\}:i=[1,t]\}$.
%\end{thm}
In this section, we prove Theorem \ref{2kr}. First we gather some results that we will use in the proof. A family $\cF$ of sets is $t$-intersecting if for any $F,F'\in\cF$ we have $|F\cap F'|\ge t$. For a set $X$ we denote by $\binom{X}{k}$ the family of all $k$-subsets of $X$. The set $\{1,2,\dots,n\}$ of the first $n$ positive integers is denoted by $[n]$, and we write $[a,b]$ for the interval $\{s\in \mathbb{N}: a\le s\le b\}$.

We will use the following theorem of Frankl.

\begin{thm}[Frankl \cite{F}]\label{fra}
Let $\cF\subseteq \binom{[n]}{k}$ be $t$-intersecting with $|\cap_{F\in \cF}F|<t$. If $n$ is large enough, then $|\cF|\le \max\{|\cF_1|,|\cF_2|\}$, where
$$\cF_1=\left\{F\in\binom{[n]}{k}:[t]\subset F, F\cap [t+1,k+1]\neq \emptyset \right\}\cup \binom{[k+1]}{k}$$
and
$$\cF_2=\left\{F\in\binom{[n]}{k}:|F\cap [t+2]|\ge t+1\right\}.$$
\end{thm}

Let us remark that later Ahlswede and Khachatrian \cite{AK} strengthened this result by determining the maximum size of a $t$-intersecting family $\cF\subset \binom{[n]}{k}$ with $|\cap_{F\in \cF}F|<t$ for any $n,k$, and $t$. 

\vskip 0.3truecm

Another tool in the proof will be the following generalization of the Erd\H os-Simonovits stability theorem \cite{E,S}. We say that two graphs $G$ and $G'$ have \textit{edit distance} at most $x$ if we can obtain $G'$ from $G$ by adding and deleting $x$ edges

\begin{thm}[Ma, Qiu \cite{MQ}]\label{mq}
Let $H$ be a graph with $\chi(H) = r + 1 > m \ge 2$. If $G$ is an $n$-vertex $H$-free
graph with $\cN (K_m,G) \ge \cN (K_m,T(n,r)) - o(n^m)$, then $G$ and $T(n,r)$ have edit distance $o(n^2)$.
%can be obtained from $T(n,r)$ by adding and deleting a set of $o(n^2)$ edges.
\end{thm}

We will also need the following two well-known results.

\begin{thm}[Zykov \cite{zykov}]\label{sym}
For any $2\le s< r\le n$ we have 
\[\ex(n,K_s,K_r)=\cN(K_s,T(n,r-1).\] 
\end{thm}

\begin{thm}[Removal lemma, Erd\H os, Frankl, Rödl \cite{efr}]\label{rem}
For any graph $H$ on $h$ vertices and $\varepsilon>0$ there exists $\delta=\delta(\varepsilon,H)$ such that if a graph $G=(V,E)$ on $n$ vertices contains at most $\delta n^h$ copies of $H$, then there exists $E'\subset E$ with $|E'|\le \varepsilon n^2$ such that $(V,E\setminus E')$ is $H$-free.
\end{thm}

\vskip 0.3truecm

\begin{proof}[Proof of Theorem \ref{2kr}]
Let $G=(V,E)$ be a $2K_r$-free graph on $n$ vertices.

Suppose first that $k< r$. Let $t$ be the minimum number such that there exists a vertex subset $U\subset V$ of size $t$ such that $G[V\setminus U]$ is $K_r$-free. Observe that $t\le r$ as if $K$ is a copy of $K_r$, then $G[V\setminus V(K)]$ must be $K_r$-free since $G$ is $2K_r$-free. By Theorem \ref{sym}, we have $\cN(K_k,G[V\setminus U])\le \cN(K_k,T(n-t,r-1))$. If $t\le 1$, then this immediately yields the statement of (i). 

Assume $2\le t\le r$. Assume that $\cN(K_k,G)\ge \cN(K_k,T(n-t,r-1))$. Then the number of $k$-cliques intersecting $U$ is at least $\cN(K_k,K_1\vee T(n-1,r-1))-\cN(K_k,T(n-t,r-1))$. Observe that 

\begin{multline*}
\cN(K_k,K_1\vee T(n-1,r-1))-\cN(K_k,T(n-t,r-1))  \\
  =(1+o(1))((t-1)\cN(K_{k-1},T(n\frac{r-2}{r-1},r-2))+ \cN(K_{k-1},T(n-1,r-1)).
\end{multline*}

If $n$ is large enough, then \[\frac{\cN(K_{k-1},T(n-1,r-1))}{\cN(K_{k-1},T(n\frac{r-2}{r-1},r-2))}>c_r>1\] for some constant $c_r$ depending only on $r$. This means that the number of $k$-cliques meeting $U$ is at least $(t-1+c_r+o(1))\cN(K_{k-1},T(n\frac{r-2}{r-1},r-2))$. As the number of $k$-cliques containing at least two vertices of $U$ is $O(n^{k-2})$, we obtain a contradiction by proving that for any $u\in U$, the number of $k$-sets $S\subset V$ with $S\cap U=\{u\}$ and $G[S]=K_k$ is at most $(1+o(1))\cN(K_{k-1},T(n\frac{r-2}{r-1},r-2))$.

To this end, observe first that for any $u\in U$ the number of $(r-1)$-subsets $A$ of $V\setminus U$ such that $\{u\}\cup A$ form a clique in $G$ is $O(n^{r-2})$. Indeed, as $U$ is minimal, there exists an $r$-clique $A'$ with $u\notin A'$, thus every such $(r-1)$-subset $A$ with $A\ni u$ must meet $A'$ (as $G$ is $2K_r$-free), so their number is at most $r\binom{n-2}{r-2}$. By Theorem \ref{rem}, there exist a set $E_0$ of $o(n^2)$ edges whose deletion removes all these $K_{r-1}$s. The number of $k$-cliques containing $u$ and at least one edge from $E_0$ is $o(n^{k-1})$. 

As the number of $k$-cliques in $G$ meeting $U$ is $O(n^{k-1})$, we must have $\cN(K_k,G[V\setminus U])\ge\cN(K_k,T(n-t,r-1))-o(n^k)$. Theorem \ref{mq} implies that $V\setminus U$ can be partitioned into $r-1$ almost parts $V_1,V_2,\dots, V_{r-1}$ such that $G[V\setminus U]$ has edit distance $o(n^2)$ from the Turán-graph on the $V_i$s as partite sets and each $V_i$ has order either $\lfloor (n-t)/(r-1)\rfloor$ or $\lceil (n-t)/(r-1)\rceil$. 

We claim that there exists an $i$ such that $|N_G(u)\cap V_i|= o(n)$.
Indeed, otherwise we can pick $\Theta(n^{r-1})$ $(r-1)$-sets having exactly one element in each $N_G(u)\cap V_i$. Only $o(n^{r-1})$ of these $(r-1)$-sets contain a pair of vertices $v,v'$ such that $vv'$ is not an edge of $G$, since only $o(n^2)$ edges between the sets $V_i$ ($i\le r-1$) are missing from $G$. Therefore, there are
%Indeed, if not, then, as $G[V\setminus U]$ is roughly complete multipartite, we would see 
$\Theta(n^{r-1})$ copies of $K_{r-1}$s in the neighborhood of $u$, but we have already seen that there are only $O(n^{r-2})$ of them. 

Clearly, there are $o(n^{k-1})$ copies of $K_k$ containing $u$ and a vertex from a $V_i$ with $|N_G(u)\cap V_i|= o(n)$. 
This implies that the number of $k$-cliques containing $u$ is indeed at most $o(n^{k-1})+\cN(K_{k-1},T(n\frac{r-2}{r-1},r-2))$ as claimed. This finishes the proof of (i).

Let us start the proof of (ii) with the special case when we only count copies of $K_k$, i.e., we are interested in $\ex(n,K_k,2K_r)$. As $r\le k<2r-1$, any two copies $K^1,K^2$ of $K_k$ must meet in at least $2k-2r+1$ vertices, otherwise their union would contain at least $2r$ vertices and thus a copy of $2K_r$. Therefore, the $k$-uniform hypergraph $H=(V,\cE)$ with $\cE=\{S\in \binom{V}{k}: G[S]=K_k\}$ is $(2k-2r+1)$-intersecting. Applying Theorem \ref{fra},  we obtain that either all the $k$-cliques of $G$ contain a fixed $(2k-2r+1)$-set, or $\cN(K_k,G)=o(n^{k-(2k-2r+1)})$. In the latter case, we are done, since $\cN(K_k,K_{2k-2r+1}\vee T(n-2k+2r-1,2r-s-1))=\Theta(n^{2r-k-1})$.

Therefore, we can assume that all the $k$-cliques of $G$ contain a fixed $(2k-2r+1)$-set $K$. Then the vertices of $K$ are adjacent to all vertices that are contained in a $K_k$ in $G$. Let $U$ denote the set of vertices outside $K$ that are contained in at least one copy of $K_k$ in $G$, so $|U|\le n-(2k-2r+1)$. If $G[U]$ is $K_{2r-k}$-free, then \begin{multline*}
    \cN(K_k,G)=\cN(K_{2r-k-1},G[U])\le \cN(K_{2r-k-1},T(n-2k+2r-1,2r-k-1))=\\
    \cN(K_k,K_{2k-2r+1}\vee T(n-2k+2r-1,2r-k-1)).
\end{multline*}. Finally, if $G[U]$ contains a copy $K'$ of $K_{2r-k}$, then $G[U\setminus [V(K')]$ cannot contain a copy of $K_{2r-k-1}$ as such a copy with $K$ and $K'$ would contain a $2K_r$. Every copy of $K_k$ in $G$ must contain $K$ and intersect $U$ in a copy of  $K_{2r-k-1}$, thus must intersect $K'$. Therefore, we have $\cN(K_k,G)=O(n^{2r-k-2})=o(\cN(K_k,K_{2k-2r+1}\vee T(n-2k+2r-1,2r-k-1))$.

Finally, let us consider the general case of (ii). Observe that similarly to the special case above, the $(k+i)$-uniform hypergraph $H_i=(V,\cE_i)$ with $\cE_i=\{S\in \binom{V}{k+i}: G[S]=K_{k+i}\}$ is $(2(k+i-r)+1)$-intersecting. Applying Theorem \ref{fra}, we obtain that the number of $(k+i)$-cliques is $O(n^{2r-k-i-1})$ and thus the number of cliques larger than $k$ is $O(n^{2r-k-2})$. This means that
\begin{itemize}
    \item 
    in order to contain $\Theta(n^{2r-k-1})$ cliques of size at least $k$, all the $k$-cliques of $G$ must contain the same $2k-r+1$ vertices just as in the special case,
    \item
    as any vertex contained in an $(k+i)$-clique is also contained in a $k$-clique, every clique of size at least $k$ is contained in $U$,
    \item
    if $G[U]$ is $K_{2r-k}$-free, then there are no cliques of size larger than $k$ in $G$, so the proof finishes as in the special case,
    \item
    if $G[U]$ does contain a $k$-clique, then, just like in the special case, there are $O(n^{2r-k-2})$ copies of $K_k$ in $G$. As the number of cliques larger than $k$ is $O(n^{2r-k-2})$, we obtain the same conclusion.
\end{itemize}
\end{proof}

\section{Forbidding $B_{r,s}$ and counting large cliques}

In this section, we prove Theorem \ref{genbrs}. Again, we begin by collecting the tools we will use. The following is a simple corollary of Theorem \ref{mq}.

\begin{proposition}\label{krr-1}
For any $r\ge 3$ and large enough $n$, we have $\ex(n,K_{r-1},K_r+ K_{r-1})=\cN(K_{r-1},T(n,r-1))$. Furthermore, if $G$ is an $n$-vertex $(K_r+K_{r-1})$-free graph with $\cN(K_{r-1},G)=\cN(K_{r-1},T(n,r-1))-o(n^{r-1})$, then $G$ has edit distance $o(n^2)$ from $T(n,r-1)$.
\end{proposition}

\begin{proof}
If an $n$-vertex graph $G$ contains a copy $K$ of $K_r$, then all the copies of $K_{r-1}$ must meet $K$, so their number is $O(n^{r-2})=o(\cN(K_{r-1},T(n,r-1))$. If $G$ is $K_r$-free, then by Theorem \ref{sym}, we have $\cN(K_{r-1},G)\le \cN(K_{r-1},T(n,r-1))$ and the furthermore part follows from Theorem \ref{mq}. 
\end{proof}

%A much studied concept in extremal finite set theory is the following: if 
If $L$ is a set of non-negative integers, we say that a family $\cF$ of sets is $L$-intersecting if for any distinct $F,F'\in \cF$, we have $|F\cap F'|\in L$. %For integers $\ell,\ell'$, and $k$ with $\ell+\ell'<k$, let $L_k(\ell,\ell')=\{0,1,\dots, \ell-1, k-\ell',k-\ell'+1,\dots,k-1\}$.

\begin{thm}[Frankl, Füredi \cite{FF}]\label{ff}
Let $\cF\subseteq \binom{[n]}{k}$ be a $\{0,1,\dots, \ell-1, k-\ell',k-\ell'+1,\dots,k-1\}$-intersecting family. Then the following statements hold.

(i) There exists a constant $d_k$ such that $|\cF|\le d_kn^{\max\{\ell,\ell'\}}$.

(ii) If $\ell'>\ell$ and $n\ge n_0(k)$, then $|\cF|\le \binom{n-k+\ell'}{\ell'}$ and equality holds if and only if there exists a $(k-\ell')$-subset $X$ of $[n]$ such that $\cF=\{F\in\binom{[n]}{k}: X\subset F\}$.

(iii) If $\ell\ge \ell'$ and $k-\ell$ has a primepower divisor $q$ with $q>\ell'$, then $|\cF|\le (1+o(1))\binom{n}{\ell}\frac{\binom{k+\ell'}{\ell'}}{\binom{k+\ell'}{\ell}}$.
\end{thm}

The result of Theorem \ref{ff} itself will not be sufficient for us, we will also need some parts of the proof.
%that relies heavily on a structural theorem of Füredi \cite{Fu}.
In the following lemma, we gather the parts of the
Frankl-Füredi proof that we will use.

To state the lemma we need to define the $i$-shadow of a family $\cF$ of sets as $\Delta_i(\cF):=\{G: |G|=i, ~\exists F\in \cF \text{ such that }  G\subset F\}$.

\begin{lemma}[Lemma 6.1 and several propositions in \cite{FF}]\label{struc}
If $\ell<\ell'$ and $\cF\subseteq \binom{[n]}{k}$ is a $\{0,1,\dots, \ell-1, k-\ell',k-\ell'+1,\dots,k-1\}$-intersecting family, then $\cF$ can be partitioned into $\cF_1\cup \cF_2\cup \cdots\cup \cF_h\cup\cF_{h+1}$ such that
\begin{itemize}
    \item 
    $|\cF_{h+1}|=O(n^{\ell'-1})$,
    \item
    for every $1\le j\le h$ there exists a $(k-\ell')$-set $A_j$ such that $\cF_j\subseteq\{G\in \binom{[n]}{k}: A_j \subset G\}$,
    \item
    writing $\cH_j=\{F\setminus A_j:F\in \cF_j\}$ we have that the $\ell$-shadows are pairwise disjoint, i.e., for every $1\le i<j\le h$  $\Delta_\ell(\cH_i)\cap \Delta_\ell(\cH_j)=\emptyset$.
\end{itemize}
\end{lemma}

We will also use the Lovász version \cite{L} of the Kruskal-Katona shadow theorem \cite{Kr,Ka}. It states that if a family $\cH$ of $k$-subsets has size $\binom{x}{k}=\frac{x(x-1)\dots (x-k+1)}{k!}$ for some real $x$, then for any $i\le k$ we have $|\Delta_i(\cH)|\ge \binom{x}{i}$. 

We will also use a theorem of Andrásfai, Erd\H os, and Sós \cite{AES} that states that an $n$-vertex $K_r$-free graph with chromatic number at least $r$ contains a vertex of degree at most $(1-\frac{1}{(r-1)-4/3})n$.

For integers $n,r,s,t$ with $r>s+t+1$ and $n>2t+s+1$, let us define the function $$f(n)=f_{r,s,t}(n)=\cN(K_{r+t},K_{s+2t+1}\vee T(n-s-2t-1,r-s-t-1))=\prod_{i=0}^{r-s-t-2}\left\lfloor \frac{n-s-2t-1+i}{r-s-t-1}\right\rfloor.$$ %and $$g(n)=g_{r,s,t}(n)=\cN(K_{r-s-t-1}, K_1 \vee T(n-s-2t-2,r-s-t-1))= \prod_{i=0}^{r-s-t-2}\lfloor \frac{n-s-2t-2+i}{r-s-t-1}\rfloor+(r-s-t-1)\prod_{i=0}^{r-s-t-2}\lfloor \frac{n-s-2t-1+i}{r-s-t-2}\rfloor.$$

Observe that for fixed $r,s,t$, the function $f(n)$ is a polynomial of $n$ of degree $r-s-t-1$. We will need the following simple properties of $f(n)$.

\begin{proposition}\label{properties}

(i) For any pair $n_1,n_2$ of positive integers $f(n_1)+f(n_2)\le f(n_1+n_2)$ holds.

(ii) For any $\varepsilon>0$ there exists $\delta>0$ such that if $n$ is large enough, then $f((1-\varepsilon)n)+f(\varepsilon n+o(n))<(1-\delta)f(n)$ holds.

(iii) If $r-s-t-1\ge 2$, $y=o(x)$ and $x=o(n)$, then $f(n-x)+f(x+y)<f(n)-\Omega(xn^{r-s-t-2})$ holds.
\end{proposition}

Now we are ready to prove Theorem \ref{genbrs}.

\begin{proof}[Proof of Theorem \ref{genbrs}]
Each of the lower bounds is obtained from the following construction: let $m=s+2t+1$ and $G$ be the join  $K_m\vee T(n-m,r-s-t-1)$. As $m+r-s-t-1=r+t$, we have $\cN(K_{r+t},G)=\cN(K_{r-s-t-1},T(n-m,r-s-t-1))=(1+o(1))(\frac{n-m}{r-s-t-1})^{r-s-t-1}$. To see that $G$ is $B_{r,s}$-free, observe that out of the $2r-s$ vertices of a copy of $B_{r,s}$, at least $2r-s-m=2r-2s-2t-1$ vertices belong to $T(n-m,r-s-t-1)$. Therefore at least $r-s-t$ vertices belong to the same $K_r$ of $B_{r,s}$ in $T(n-m,r-s-t-1)$. As there is no clique of $r-s-t$ vertices in $T(n-m,r-s-t-1)$, the graph $G$ is indeed $B_{r,s}$-free.

For the general upper bound, let $G$ be a $B_{r,s}$-free graph on $n$ vertices. Define the $(r+t)$-uniform family $\cF_G=\{K\subseteq \binom{V(G)}{r+t}: G[K]=K_{r+t}\}$. Observe that $\cF_G$ is $\{0,1,\dots, s-1, 2t+s+1,2t+s+1,\dots,r+t-1\}$-intersecting. 
%$L_{r+t}(s,r-t-s-1)$ 
Indeed, assume that two cliques $K_1,K_2$ each of size $r+t$ intersect in at least $s$, but less than $(r+t)-(r-t-s-1)=2t+s+1$ vertices. Then the union of $K_1$ and $K_2$ contains at least $2(r+t)-(2t+s)=2r-s$ vertices and their intersection contains at least $s$ vertices, and thus $G$ contains a copy of $B_{r,s}$, a contradiction. We can apply Theorem \ref{ff} (i) to show that $\cN(K_{r+t},G)=O(n^{r-s-t-1})$.
%By the assumption $2s+t+1\le r$, we have $\ell\le \ell'$, and thus by Theorem \ref{ff} (i), we have $\cN(K_{r+t},G)=O(n^{r-s-t-1})$.

Finally, we consider the case $s=1$. Let $k=r+t$, and $\ell'=r-t-2$. Suppose $G$ is such that $\cN(K_{r+t},G)\ge f(n)$. By the assumption $t+3<r$, we have $1<\ell'$ and so by Lemma \ref{struc},  we obtain a partition $\cF_G=\cF_1\cup \cdots\cup \cF_h\cup\cF_{h+1}$ and $m$-sets $A_1,A_2,\dots,A_h$ with the properties ensured by Lemma \ref{struc}. We introduce positive reals $x_1,\dots, x_h$ such that $|\cF_i|=|\cH_i|=\binom{x_i}{\ell'}$. Without loss of generality, $x_1\ge x_2\ge\cdots\ge x_h$. Let $M_j=\cup_{
F\in \cF_j}F$ and  clearly, we have $|M_j|\ge x_j$.

\begin{clm}\label{klem}
There exists an integer $n_0$ and a constant $C$ such that if $n\ge n_0$, then $\Delta_1(\cF_1)\ge n-C$.
\end{clm}

\begin{proof}
By the Lovász version of the Kruskal-Katona shadow theorem, we know that $|\Delta_\ell(\cH_i)|\ge x_i$ and thus by Lemma \ref{struc} we have $\sum_{i=1}^hx_i\le n$. Also, $x_i\ge x_j$ implies $\frac{\binom{x_i}{\ell'}}{x_i}\ge \frac{\binom{x_j}{\ell'}}{x_j}$. Therefore, 
\begin{equation}\label{eq}
\sum_{j=i}^h|\cF_j|=\sum_{j=i}^h|\cH_j|=\sum_{j=i}^hx_j\frac{\binom{x_j}{\ell'}}{x_j}\le \frac{\binom{x_i}{\ell'}}{x_i}\sum_{j=i}^hx_j\le \frac{\binom{x_i}{\ell'}}{x_i}n=O((x_i/n)^{\ell'-1}n^{\ell'}).
\end{equation}
Then $|M_1|\ge \varepsilon n$ for some fixed $\varepsilon>0$ as otherwise $|\cF|=o(n^{\ell'})$, while $f(n)=\Theta(n^{\ell'})$.
 By Lemma \ref{struc}, we have that the sets $M_j\setminus A_j$ are pairwise disjoint and thus $|M_j\setminus \cup_{j'=1}^{j-1}M_{j'}|\ge |M_j|-(j-1)m$. Let $j_1$ be the largest index $j$ with $|M_j|\ge (m+1)n^{2/3}$. Then (\ref{eq}) implies that $\sum_{j=j_1+1}^h|\cF_j|=O(n^{\ell'-1/3})$, and Lemma~\ref{struc} shows that $|\cF_{h+1}|=O(n^{\ell'-1})$. We claim that $j_1\le n^{1/3}$. Indeed, if  not, then for $j\le n^{1/3}$, we have $|M_j\setminus \cup_{j'=1}^{j-1}M_{j'}|\ge |M_j|-jm> n^{2/3}$ 
 and $n\ge \sum_{j=1}^{n^{1/3}}|M_j\setminus \cup_{j'=1}^{j-1}M_{j'}|> n^{1/3}n^{2/3}$, a contradiction. As a consequence, we also have \[\sum_{j=1}^{j_1}|M_j|\le n+\sum_{j=1}^{j_1}(j-1)m\le n+j_1^2m=n+O(n^{2/3}).\]
 
 Observe that for every $j\le h$, every vertex of $M_j$ is connected to every vertex of $A_j$. This implies that $G[M_j\setminus A_j]$ is $K_{r-t-1}+K_{r-t-2}$-free, and thus, by Proposition \ref{krr-1}, $|\cF_j|=\cN(K_{r+t},G[M_j])\le f(|M_j|)$. Using Proposition \ref{properties} (i) and (\ref{eq}) we obtain \begin{equation}\label{eqnew}
\sum_{j=1}^{h+1}|\cF_j|\le f(|M_1|)+f(n+j_1^2m-|M_1|)+O(n^{\ell'-1/3})+O(n^{\ell'-1}).   
 \end{equation}
 Assume first that $|M_1|< n-n^{2/3}\log n$. Let $x=n^{2/3}\log n-j_1^2$ and $y=j_1^2$. Then Proposition  \ref{properties} (iii) yields that $f(|M_1|)+f(n+j_1^2-|M_1|)<f(n)-\Omega(xn^{\ell'-1})$. Therefore, the right hand side of (\ref{eqnew}) is less than $f(n)-\Omega(xn^{\ell'-1})+O(n^{\ell'-1/3})+O(n^{\ell'-1})<f(n)$, a contradiction.
 
 Suppose towards a contradiction that $r(n):=n-|M_1|$ tends to infinity. Let us write $p(n):=(m+1)(r(n))^{2/3}$ and let $j^*$ be the largest index with $|M_{j^*}|\ge p(n)$. Then by (\ref{eq}), we obtain $\sum_{j=j^*+1}^{h+1}|\cF_j|=O(n(p(n))^{\ell'-1})=O(n^{\ell'-1}p(n))$. We claim that $j^*\le (r(n))^{1/3}$. Indeed, we can argue similarly as for the upper bound on $j_1$ earlier: if $(r(n))^{1/3}\le j^*$, then for $j\le (r(n))^{1/3}$ we have $|M_j\setminus \cup_{i=1}^{j-1}M_i|\ge |M_j|-j(1+m)\ge (r(n))^{2/3}$ and $r(n)\ge \sum_{j=2}^{(r(n))^{1/3}}|M_j-\cup_{j'=1}^{j-1}M_{j'}|> (r(n))^{1/3}(r(n))^{2/3}$, a contradiction.
 This implies that $\sum_{j=2}^{j^*}|M_j|\le r(n)+(1+m) p(n)$. Applying Proposition \ref{properties} (i) and (iii) with $x=r(n)$ and $y=(1+m) p(n)$, we obtain that 
\begin{multline*}
 \cN(K_{r+t},G)=\sum_{j=1}^{h+1}|\cF_j|\le f(|M_1|)+\sum_{j=2}^{j^*}f(M_j)+O(n^{\ell'-1}p(n))\le \\ f(n-r(n))+f(r(n)+(1+m)p(n))+O(n^{\ell'-1}p(n))<f(n),  
\end{multline*}
 a contradiction.
\end{proof}

By Claim \ref{klem}, we may assume that $|M_1|\ge n-C$ for some constant $C$. By Lemma \ref{struc}, this implies $\cN(K_{r+t},G)\le f(|M_1|)+O(n^{\ell'-1})$. As $\cN(K_{r+t},G)\ge f(n)$, we must have $\cN(K_{\ell'},G[M_1\setminus A_1])\ge f(|M_1|)-D(n^{\ell'-1})$ for some constant $D$. Proposition \ref{krr-1} implies that $G[M_1\setminus A_1]$ is $K_{\ell'+1}$-free and it can be made $\ell'$-partite by deleting $o(n^2)$ edges. Let us delete those edges and let $U_1,U_2,\dots,U_{\ell'}$ be the corresponding partition. We say that a vertex $v\in U_i$ is \textit{problematic} if there exists $j\neq i$ such that there are at least $\frac{|U_j|}{(\ell')^2}$ vertices in $U_j$ not adjacent to $v$. A set of vertices $W\subset M_1\setminus A_1$ is \textit{good} if it does not contain any problematic vertices.
\begin{clm}
There exists a set $X\subset M_1\setminus A_1$ with $|X|=O(1)$ such that $M_1\setminus (A_1\cup X)$ is good.
\end{clm}
\begin{proof}[Proof of Claim.] First we find a set $X_1$ of vertices whose removal from $G[M_1\setminus A_1]$ makes the remaining graph $\ell'$-partite such that $|X_1|=O(1)$. Then we show that there are $O(1)$ problematic vertices in the remaining $l'$-partite graph.
%We construct $X$ as $X_1 \cup X_2$. The removal of the vertices of $X_1$ will make the remaining graph $\ell'$-partite, while $X_2$ will remove all the problematic vertices. 

Suppose first that $\chi(G[M_1\setminus A_1])\ge \ell'+1$. Then by a theorem of Andrásfai, Erd\H os, and Sós \cite{AES}, there exists a vertex $v$ with degree (in $G[M_1\setminus A_1]$) at most $(1-\frac{1}{\ell'-4/3}+o(1))|M_1\setminus A_1|$. As $G[M_1\setminus A_1]$ is $K_{\ell'+1}$-free, $G[N_G(v)\cap (M_1\setminus A_1)]$ is $K_{\ell'}$-free and the number of copies of $K_{r+t}$ in $G[M_1]$ containing $v$ is at most $$\cN(K_{\ell'-1},G[N_G(v)\cap (M_1\setminus A_1))\le \cN(K_{\ell'-1},T((1-\frac{1}{\ell'-4/3}+o(1))|M_1\setminus A_1|,\ell'-1)).$$
Now observe that the difference between the number of copies of $K_{\ell'}$ in $T_{\ell'}(|M_1\setminus A_1|)$ containing a fixed vertex $u$ and this number is at least $\alpha n^{\ell'-1}$ for some constant $\alpha$. So we remove $v$ and add it to $X_1$. If the remaining graph is $\ell'$-partite, then we are done with the first step, otherwise we use the Adrásfai, Erd\H os, Sós theorem to find another vertex of low degree, and so on. Observe that if $|X_1|\alpha$ is larger than $D$, then $\cN(K_{r+t},G)\le f(n)$, so indeed we can guarantee that the size of $X_1$ is bounded by a constant.

From now on, we can assume that the remaining graph $G[M_1\setminus A_1\setminus X_1]$ is $\ell'$-chromatic with partition $U_1,\dots, U_{\ell'}$. If a vertex $u\in U_i$ is problematic, then the number of copies of $K_{r+t}$ in $G[M_1]$ containing $v$ is at most $(1-\frac{1}{(\ell')^2})\prod_{j\neq i}|U_j|$, so again some $\beta n^{\ell'-1}$ smaller than in the appropriate Tur\'an graph. We remove problematic vertices one by one, let $X_2$ be the set of vertices removed this way. As in the above paragraph, if $|X_2|\beta$ is larger than $D$, then $\cN(K_{r+t},G)\le f(n)$, so indeed we can guarantee that the size of $X_2$ is bounded by a constant.
\end{proof}

We claim that if an $(r+t)$-clique $W$ contains a vertex from $V \setminus M_1$, then $W$ and $M_1\setminus X$ are disjoint. 
Indeed, assume to the contrary that for a clique $W\not\subset M_1$ we have $|W\cap (M_1\setminus X)|\ge 1$ and let $y$ be an element of $W\cap A_1$ if such an element exist and $y\in W \cap M_1\setminus (A_1\cup X)$ otherwise. 

Let us go through the indices $i$ with $y\notin U_i$ in an arbitrary order. For each $i$, we pick a vertex $v_i\in U_i \setminus W$ that is connected to $y$ and every vertex already picked. As the  number of vertices picked this way is at most $r-t-3$, at most $(r-t-3)|U_i|/(r-t-2)^2+o(n)$ vertices of $U_i$ are forbidden, thus we can pick the desired vertex. Then we can add the vertices of $A_1\cup |\{y\}$ to obtain a clique $W'$ of size $m+\ell'=r+t$. Because $y$ is in both $W$ and $W'$, we have $|W\cap W'|\ge 1$. We claim that $|W\cup W'|\ge m+2\ell'+1=2r-1$. Indeed, as $W$ contains a vertex from $V\setminus M_1$, we cannot have $A_1\subset W$ and thus by construction, we have $|W\cap W'|<m$. Observe that $1\le |W\cap W'|$ and $|W\cup W'|\ge 2r-1$ imply that $W\cup W'$ contains a copy of $B_{r,1}$. This contradiction shows that $W$ is indeed disjoint with $M_1\setminus X$.

The number of $(r+t)$-cliques disjoint with $M_1\setminus X$ is $$\binom{|(V\setminus M_1)\cup X|}{r+t}=O(1)=O(n^{\ell'-2}).$$

As a consequence we obtain that the number of $(r+t)$-cliques of $G$ meeting $V\setminus M_1$ is $o(n^{\ell'-1})$, while $f(n)-f(n-C)=\Omega(n^{\ell'-1})$ as long as $C$ is positive. Therefore, we must have that $|M_1|=n$. As the sets $M_i\setminus A_i$ are pairwise disjoint, this implies that for any $F\in \cF_j$, the subset $F\setminus A_j$ be contained in $M_1\setminus A_1$, so their total number is $O(n^m)$. That means even with the exceptional sets from $\cF_{h+1}$, the number of sets of $\cF_G\setminus\cF_1$ is $O(n^{\ell'-1})$. If there exists a set $F\in \cF_G$ with $A_1\not\subset F$ and $A_1\cap F\neq \emptyset$, then, as $G$ is $B_{r,1}$-free, every $F_1\in \cF_1$ must meet $F$ outside $A_1$, so $\cF_1$ and $\cF$ both have size $O(n^{\ell'-1})$, which contradicts $|\cF|\ge f(n)$. If $\cF$ contains a set $F$ disjoint with $A_1$, then, as $G$ is $B_{r,1}$-free, every set $F_1\in \cF_1$ is either disjoint with $F$ or $|F\cap F_2|\ge 2$, so $|\cF|\le f(n-S)+\binom{n}{\ell'-2}+|\{F\in \cF:F\cap A_1=\emptyset\}$, where $S=|\cup_{F\in\cF, |A_1\cap F|=0}F|$.  By Proposition \ref{properties} (iii),  $f(n)-f(n-S)=\Omega(n^{\ell'-1})$, and so $|\cF|<f(n)$, a contradiction.

%Now if $S$ is bounded by a constant, then by Proposition \ref{properties} (iii),  $f(n)-f(n-S)=\Omega(n^{\ell'-1})$, and so $|\cF|<f(n)$. If $S$ tends to infinity, then again by Proposition \ref{properties} (iii), $f(n)-f(n-S)=\omega(n^{\ell'-1})$, and again $|\cF|<f(n)$.

We have established that all $F\in \cF_G$ must contain $A_1$, and then we are done by Proposition \ref{krr-1}.
\end{proof}

\section{Forbidding $B_{r,1}$ and counting small cliques}

In this section we prove Theorem \ref{newn}. Recall that it states that for $k<r$ and $n$ large enough, 
$\ex(n,K_k,B_{r,1})=\cN(K_k,T^+(n,r-1))$. We will use the asymptotic result $\ex(n,K_k,B_{r,1})=(1+o(1))\cN(K_k,T^+(n,r-1))=(1+o(1))\cN(K_k,T(n,r-1))$. It follows from a theorem of Alon and Shikhelman \cite{AS} that states that if $\chi(F)=r>k$, then $\ex(n,K_k,F)=(1+o(1))\cN(K_k,T(n,r-1))$.

\begin{proof} Let $G$ be a $B_{r,1}$-free graph on $n$ vertices.
By Theorem \ref{mq}, if $G$ has at least $\cN(K_k,T(n,r-1))-o(n^k)$ copies of $K_k$, then $G$ can be obtained from $T(n,r-1)$ by adding and removing at most $\varepsilon n^2$ edges. We consider the $(r-1)$-partite subgraph $G'$ of $G$ with the most number of edges. Let $V_1,\dots,V_{r-1}$ be the parts of $G'$, then there are $o(n^2)$ edges inside the parts $V_i$. Moreover, each vertex is connected to at most as many vertices in its part as in any other part. Also, every $V_i$ has size $(1-o(1))\frac{n}{r-1}$.

Let us pick $\alpha<(r-2)/(r-1)$, and assume first that every vertex has degree at least $\alpha n$. We partition $V_i$ to $V_i'$ and $V_i''$ with $V_i'$ containing those vertices of $V_i$ that are connected to all but $o(n)$ vertices outside $V_i$. By the assumption on the degrees, all $v\in V_i''$ are incident to $\Omega(n)$ edges inside $V_i$. This implies $|V_i''|=o(n)$. 

We will use that for any $i$, for any set $U\subset V_1'\cup \dots, V_i'$ with $|U|=O(1)$, the common neighborhood of the vertices of $U'$ contains all but $o(n)$ vertices from $V_{i+1}'\cup \dots\cup V'_{r-1}$. In particular, if we take at most four vertices from $V_1$, we can find a copy $B_{r,1}$ or $B_{r,0}$ in their common neighborhood by picking the necessary vertices from the other parts one by one.
Let $m$ be the largest number of edges inside a $V_i$, without loss of generality there are $m$ edges inside $V_1$.

Consider first the case that $m>1$. We claim that there cannot be two edges within $V_1'$. Indeed, if $u,v,w$ is a path, then in the common neighborhood of $u,v,w$ outside $V_1$, we can find two disjoint cliques $K_1,K_2$ each of size $r-2$, so $G[K_1,\cup K_2, \cup\{u,v,w\}]$ contains a copy of $B_{r,1}$. Similarly, if $uv$ and $wz$ are edges in $V_1$, then in the common neighborhood of $u,v,w,z$ outside $V_1$ one can find a copy $B$ of $B_{r-2,1}$, and $B$ together with $u,v,w,z$ form a copy of $B_{r,1}$. These contradictions prove our claim.

%Observe that, by the assumption on the degrees, a vertex $u\in V_1$ is either incident to $\Omega(n)$ edges inside $V_1$ or is adjacent to all but $o(n)$ vertices outside $V_1$. If there are three vertices $u,v,w\in V_1$ such that there is an edge between them and they are each connected to all but $o(n)$ vertices outside $V_1$, then there is no $2K_{r-2}$ in the common neighborhood of $u$, $v$ and $w$.  We can now count the copies of $K_k$. There are $o(n^k)$ copies containing an edge inside a $V_i$. There are $o(n^k)$ copies containing a vertex outside $V_1$ that is not in the common neighborhood of $u$, $v$ and $w$. The other copies contain a vertex from $V_1$, at most $\lceil n/(r-1)\rceil$ ways, and a $K_{k-1}$ from the common neighborhood of $u$, $v$ and $w$, at most $\ex(\lceil (r-2)n/(r-1)\rceil,K_{k-1},2K_{r-2})$ ways. Adding up completes the proof of the claim. \pb{igazibol itt az igaz, hogy $V_1$-ben nincs ket el, ami majdnem mindenkihez bekotott csucsok kozott menne, mert akkor talalunk $B_{r,1}$-et}

The remaining possibility is that %the edges inside $V_1$ are covered by stars of order at least $\Omega(n)$, and the center of each star is connected to $|V_2\cup \dots \cup V_{r-1}|-\Omega(n)$ vertices outside $V_1$. Let $u$ be the center of such a star, connected to $n'$ vertices inside $V_1$. 
there exists a vertex $u\in V_1''$ and thus $u$ has $n'=\Omega(n)$ neighbors in $V_1$. We also know that for any $i>1$, $u$ has at least $n'$ neighbors in $V_i$, thus at least $n'-o(n)\ge n'/2$ neighbors in $V_i'$. Let $U_i$ be an arbitrary set of $\lceil n'/2\rceil$ neighbors of $u$ in $V_i$, and $U=U_2\cup \dots, U_{r-1}$.
%Let $U$ be the neighborhood of $u$ in $V'_2\cup V'_3\cup \dots \cup V'_{r-1}$, then $|U\cap V'_i|\ge n'-o(n)$ for every $i$.  
We can now count the copies of $K_k$. There are $o(n^k)$ copies containing an edge inside a $V_i$. Let $G''$ denote the complete $(r-1)$-partite graph with parts $V_1,\dots,V_{r-1}$.
Now compare the number of copies of $K_k$ inside $G'$ to the number of copies of $K_k$ in $G''$. We claim that that $U$ is $2K_{r-2}$-free. Indeed, if $K_1,K_2$ were two cliques of each of size $r-2$ in $U$, then the common neighborhood of the of the vertices of $K_1\cup K_2$ would contain all but $o(n)$ vertices in $V_1$, in particular all but $o(n)$ neighbors of $u$. So $K_1,K_2,u$ and two such neighbors would form a $B_{r,1}$ in $G$. This contradiction proves our claim.

The $2K_{r-2}$-free property yields that the number of copies of $K_k$ containing a vertex of $V_1$ and a $K_{k-1}$ inside $U$ is less in $G'$ by $\Theta(n^k)$ than in $G''$. Indeed, there are at most
$|V_1|\ex(|U|,K_{k-1},2K_{r-2})=(1+o(1))|V_1|\cN(K_k,T(|U|,r-3))$ such copies of $K_k$ in $G'$ and $|V_1|\cN(K_k,T(|U|,r-2))$ such copies of $K_k$ in $G''$.
This implies that $G$ has less copies of $K_k$ than a complete $(r-1)$-partite graph, which has no more copies of $K_k$ than the Tur\'an graph, completing the proof in this case.

Assume now that $m=1$, i.e. there is at most one edge inside each $V_i$. If there are at least two such edges, say $uv\in V_1$ and $xy\in V_2$, then we pick $U=\{u_1,u_3,\dots,u_{r-1}\}$ and $U'=\{v_2,u_3,v_4,\dots,v_{r-1}\}$ with $u_i,v_i\in V_i$. If both $U\cup \{x,y\}$ and $U'\cup\{u,v\}$ induce cliques, then we find $B_{r,1}$, a contradiction. Thus there is a missing edge between parts $V_i$ among these vertices, thus there are $\Omega(n)$ missing edges between parts altogether. We can again compare the number of copies of $K_k$ inside $G'$ to the number of copies of $K_k$ in the complete $(r-1)$-partite graph with parts $V_1,\dots,V_{r-1}$. The additional at most $r-1$ edges inside parts create $O(n^{k-2})$ copies of $K_k$, while the missing edges between parts are in $\Omega(n^{k-1})$ copies of $K_k$.

If there is only one edge $uv$ inside a part, say $V_1$, then we have to show that the order of the parts is as balanced as possible. 
%for the parts $V_2,\dots,V_{r-1}$ it follows from Zykov's theorem, as every $K_k$ intersects those parts in a clique of order $k-2$ or $k-1$ or $k$. 
If we have $|V_i|\ge |V_j|+2$, then we move a vertex from $V_i$ to $V_j$. The number of copies of $K_k$ not containing $uv$ increases by $\Theta(n^{k-2})$, while the number of copies of $K_k$ containing $uv$ increases by $O(n^{k-3})$, completing the proof in this case.

Assume now that there are vertices of degree less than $\alpha n$. We erase such vertices one by one till we arrive to a graph $G_0$ with no such vertices. Assume that we removed $\ell$ vertices. If $\ell<n/2$, then $G_0$ has sufficiently many vertices, thus at most $\cN(K_k,T^+(n-\ell,r-1))$ copies of $K_k$ by the above part of the proof. We removed at most $\ell \ex(\alpha n,K_{k-1},2K_{r-1})=(1+o(1))\ell\cN(K_{k-1},T(\alpha n,r-2))$ copies of $K_k$, using Theorem \ref{2kr}. If we add $\ell$ vertices to $T^+(n-\ell,r-1))$ to form the $T^+(n,r-1))$ instead, then we add $(1+o(1))\ell\cN(K_{k-1},T((r-2)n/(r-1),r-2))$ copies of $K_k$, thus we obtain more copies than in $G$, completing the proof.

If $\ell>n/2$, then we cannot apply the earlier part of the proof, since $|V(G_0)|$ may not be large enough. However, we removed at most $\ell \ex(\alpha n,K_{k-1},2K_{r-1})=(1+o(1))\ell\cN(K_{k-1},T(\alpha n,r-2))=\ell\cN(K_{k-1},T((r-2)n/(r-1),r-2))-\Theta(n^k)$ copies of $K_k$, and the resulting graph $G_0$ has at most $(1+o(1))\cN(K_k,T^+(n-\ell,r-1))$ copies of $K_k$ by the known asymptotic bound, completing the proof.

\end{proof}

\section{Forbidding $B_{3,1}$ and counting complete bipartite graphs}

Let us start this section by describing the symmetrization method due to Zykov \cite{zykov}. He used it to show that $\ex(n,K_k,K_r)=\cN(K_k,T(n,r-1))$. We say that we \textit{symmetrize} a vertex $u$ to another vertex $v$ in a graph $G$ when we delete all the edges incident to $u$ and for each edge $vw$, we add the edge $uw$. In other words, we replace the neighborhood of $u$ by the neighborhood of $v$. We apply this operation to non-adjacent vertices.
One can show that if $G$ is $K_r$-free for some $r$, then the graph $G'$ we obtain by symmetrizing $u$ to $v$ is also $K_r$-free.

Let $d(H,v)$ denote the number of copies of $H$ containing a vertex $v$. Extending Zykov's idea, Gy\H ori, Pach and Simonovits \cite{gypl} showed that if $H$ is a complete multipartite graph and $d(H,u)\le d(H,v)$, then this symmetrization does not decrease the total number of copies of $H$. Thus, for any pair of non-adjacent vertices $(u,v)$ we can symmetrize one to the other such that the total number of copies of $H$ does not decrease. We apply such symmetrization steps as long as we can find two non-adjacent vertices with different neighborhoods. At the end of the symmetrization process we obtain a $K_r$-free complete multipartite graph with at least $\cN(H,G)$ copies of $H$, which implies that $\ex(n,H,K_r)$ is attained by a complete $(r-1)$-partite graph (one also needs to show that this process terminates).

In a sense, this is the most general application of Zykov symmetrization for generalized Tur\'an problems: if $F$ is any graph that is not a clique, then symmetrization may ruin the $F$-free property. If $H$ is any graph that is not complete multipartite, then it is possible that both symmetrizing $u$ to $v$ and symmetrizing $v$ to $u$ decreases the total number of copies of $H$. However, Liu and Wang \cite{LW} introduced a restricted version of symmetrization that avoids the first of these problems. Here we state the general version of the basic idea.

\begin{proposition}
Let $G$ be a $B_{r,s}$-free graph, $u$ and $v$ be non-adjacent vertices of $G$, and assume that $v$ is not a rootlet vertex of any $B_{r,s+1}$ in $G$. Let $G'$ be the graph obtained from $G$ by symmetrizing $u$ to $v$. Then $G'$ is $B_{r,s}$-free. 
\end{proposition}

\begin{proof}
Assume that there is a copy $B$ of $B_{r,s}$ in $G'$. Then $B$ has to contain $u$, otherwise $B$ would be contained in $G$. If $B$ does not contain $v$, then we can replace $u$ with $v$ to obtain a copy of $B_{r,s}$ that is also present in $G$, a contradiction. If $B$ contains $v$, then, as $u$ and $v$ are not adjacent in $G'$, they are both page vertices of $B$, on different pages. But they have the same neighborhood, thus $v$ is connected to every vertex of $B$ but $u$. Then the $s$ rootlet vertices of $B$ with $v$ form the rootlet vertices of a copy $B'$ of $B_{r,s+1}$ in $G$, where the pages of $B'$ are the pages of $B$ without $u$  and $v$. Thus $v$ is a rootlet vertex of a $B_{r,s+1}$ in $G$, contradicting our assumption.
\end{proof}

Let us repeatedly apply symmetrization on the vertices that are not rootlet vertices of any $B_{r,s+1}$. Assume that the process terminates and let $G_0$ be the resulting graph. Let $G_1$ be the subgraph of $G_0$ induced by vertices that are not rootlet vertices of any $B_{r,s+1}$ in $G_0$. Then the proposition above implies that $G_1$ is a complete multipartite graph. Moreover, vertices of the same partite set of $G_1$ have the exact same neighborhood in the other vertices of $G_0$.

%The proposition above implies that if we repeatedly apply symmetrization on the vertices that are not rootlet vertices of any $B_{r,s+1}$, then the subgraph of the resulting graph $G_0$ induced by vertices that are not rootlet vertices of any $B_{r,s+1}$ in $G_0$ is a complete multipartite graph (assuming the process terminates). Moreover, vertices of the same partite set have the exact same neighborhood in the other vertices of $G_0$ as well.

Now we are ready to prove Theorem \ref{b31} which we restate for convenience.

\begin{theorem*}
%If $H$ is a complete bipartite graph, then
For any integers $a\le b$ and $n$ large enough, we have that $\ex(n,K_{a,b},B_{3,1})=\cN(K_{a,b},T)$ for some $n$-vertex graph $T$ that is obtained from a complete bipartite graph by adding an edge.
\end{theorem*}

%We remark that in the case $K_{a,b}$ is not a star, the extra edge cannot be in any copy of $K_{a,b}$, thus $K_{a,b}$ is weakly $B_{3,1}$-Tur\'an-good.

\begin{proof}
Let $G$ be an $n$-vertex $B_{3,1}$-free graph with $\ex(n,H,B_{3,1})$ copies of $H=K_{a,b}$. Let $Q$ denote the set of vertices in $G$ that are not rootlet vertices of a $B_{3,2}$. If there are two non-adjacent vertices $u$ and $v$ in $Q$ with $d(H,u)<d(H,v)$, then we symmetrize $u$ to $v$ and obtain a $B_{3,1}$-free graph with more copies of $H$, a contradiction. Thus we can assume that non-adjacent vertices in $Q$ have the same $d(H,v)$ value (and also later, after any symmetrization). This means that for non-adjacent vertices $u$ and $v$, we can choose whether we symmetrize $u$ to $v$ or $v$ to $u$ and the total number of copies of $H$ will not change. 

Recall that $B_{3,2}$ consists of two triangles sharing an edge. We call the graph consisting of $k\ge 2$ triangles sharing an edge a \textit{book graph with $k$ pages}. Observe that for any vertex $v\in Q$ there exists at most one book $B_v$ of which $v$ is a page vertex. Indeed, the two books have at least 3 rootlet vertices. If the two books have four rootlet vertices, we find a $B_{3,1}$. Thus the two books have at least 3 rootlet vertices $x,y,z$, each connected to $u$, such that $xy$ and $yz$ are in $E(G)$. If these three vertices and $v$ induce a $K_4$, then $v$ is a rootlet vertex of a $B_{3,2}$, a contradiction. Otherwise, there is a fifth vertex $w$ in one of the books, without loss of generality $w$ is joined to $x$ and $y$ with an edge. Then $x,y,w$ and $y,z,v$ form triangles sharing exactly one vertex, a contradiction.

%at least 3 rootlet vertices of the two books together with the page vertices would contain a $B_{3,2}$ unless $v$ is the only page vertex and the four vertices induce a $K_4$, which contradicts the assumption that $v\in Q$ is not a rootlet vertex of any $B_{3,2}$.

Let $u$ and $v$ be page vertices from different copies of $B_{3,2}$ that are not rootlet vertices of any $B_{3,2}$. If $B_u$ has more pages than $B_v$, then we symmetrize $v$ to $u$. If they have the same number of pages, then we symmetrize arbitrarily.
If $u$ is the page vertex of a $B_{3,2}$ but $v$ is not, then we symmetrize $v$ to $u$. We claim that after such a symmetrization, 
%the total number of pages of book graphs does not decrease, while 
the total number of pages in the largest $i$ books does not decrease for any $i$ and increases for some $i$. Indeed, observe that if $B_v$ is among the largest $i$ books, then so is $B_u$. We deleted at most one page containing $v$, but added a page when we connected $v$ to the neighbors of $u$. This proves the first part of the claim. Assume that there are $i-1$ books with more pages than $B_u$. Those are unaffected by the symmetrization, and $B_u$ has one more page, proving the second part of our claim. 

Let us apply such symmetrization steps as long as we can. We claim that after finitely many steps this process terminates. Indeed, for any $i$ the total number of pages in the largest $i$ books can increase at most $in$ times.

Let $G_0$ be the resulting graph, $G_1$ be the subgraph induced on the vertices that are not rootlet vertices of any $B_{3,2}$ and $G_2$ be the subgraph induced by the other vertices. Note that the vertex set of $G_1$ might be different from $Q$, as symmetrization may destroy or create copies of $B_{3,2}$. Non-adjacent vertices of $G_1$ have the same neighborhood, as otherwise we would symmetrize one vertex to another. This means that $G_1$ is a complete $m$-partite graph for some $m$ with partite sets $A_1,\dots,A_m$. Observe that $m\le 3$ because there are no rootlet vertices in $G_1$. 

%Innen még nem világos hogyan fejezzük be...

Observe that in a page vertex of a $B_{3,2}$ can be a rootlet vertex of a $B_{3,2}$ only if the book is actually a $K_4$, as otherwise one can easily find a $B_{3,1}$. Furthermore, every vertex outside the $K_4$ can have at most one neighbor inside that $K_4$. This implies that for every $i\le m$, we have that the vertices of $A_i$ have at most one neighbor in a $K_4$. If $u$ and $v$ are the rootlet vertices of a $B_{3,2}$ that is not a $K_4$, then its page vertices are in $G_1$, and belong to the same partite set $A_i$. It is easy to see that if $u'v'$ is an edge in $G_2$ and $\{u',v'\}\neq \{u,v\}$, then we cannot have
that $u'$ and $v'$ are both connected to a vertex $w\in A_i$. Indeed, if $u,v,u',v'$ are four vertices, then they form $B_{3,1}$ with $v$. If, say, $u=u'$, then $u,v',w$ and $u,v,w'$ for any page vertex $w'$ from $B_w$ form two triangles sharing exactly one vertex.
In other words, the only copies of $B_{3,2}$ with page vertices in $A_i$ are those with rootlet vertices $u$ and $v$.

Assume first that $m=3$. Then $G_1$ is a triangle because there are no rootlet vertices in $G_1$.
Then $G_0$ is the vertex-disjoint union of a $K_3$ and several copies of $K_4$, with additional edges between these subgraphs. However, every vertex $v$ is connected to at most one vertex in every copy of $K_4$ and $K_3$ (except the one containing $v$). This means that the degree of $v$ is at most $d:=2+(n+1)/4$. Then we can count the copies of $K_{a,b}$ the following way. Assume that $a\neq b$; if $a=b$, then each of our bounds are divided by 2 because of symmetry. Either we pick a vertex $v$, $a$ of its neighbors and $b-1$ of their common neighbors other than $v$, or pick $b$ neighbors of $v$ and $a-1$ of their common neighbors other than $v$. This way we count every copy of $K_{a,b}$ exactly $a+b$ times, since we can pick any vertex as the first vertex. Thus we obtain the upper bound $n(\binom{d}{a}\binom{d}{b-1}+\binom{d}{b}\binom{d}{a-1})/(a+b)$. We can count the copies of $K_{a,b}$ in $K_{\lfloor n/2\rfloor,\lceil n/2\rceil}$ the same way and we get $n(\binom{d'}{a}\binom{d'}{b-1}+\binom{d'}{b}\binom{d'}{a-1})/(a+b)$ copies, where $d'=\lfloor n/2\rfloor$. Clearly this is a larger number if $n$ is large enough, a contradiction to our assumption that $G$ is extremal.
%Observe that $G_0$ contains $K_{a,b}$ only if $a=b=2$ or $a=1$, thus we are done in the other cases. The case $a=b=2$ was proved by Gerbner and Palmer \cite{}. For the case $a=1$, 
%Observe that every degree of $K_{\lfloor n/2\rfloor,\lceil n/2\rceil}$ is larger than any degree of $G_0$, ... it contains more copies of $K_{a,b}$ than $G_0$, a contradiction... Sőt ez ha $n/3$-nál kevesebb pont van $G_1$-ben...

Assume now that $m\le 2$. We will handle together the cases $m=1$ and $m=2$; in what follows $A_2$ may be empty.
Observe that $G_2$ contains at most two edges $uv$ and $u'v'$ such that $u,v$ are connected to the vertices in $A_1$ and $u',v'$ are connected to the vertices of $A_2$. $G_2$ also may contain several copies of $K_4$, with additional edges between these subgraphs. However, at most one vertex of those copies of $K_4$ can be connected to vertices in $A_i$ for every $i$. Also, such a vertex can be connected to vertices of at most one of $A_1$ and $A_2$ otherwise $G_0$ would contain a $K_4$ and a $K_3$ meeting in exactly one vertex, thus a $B_{3,1}$.
%In particular, there are at most two vertices of the $K_4$ that are connected to vertices in $G_1$. 
%Observe that if $a\ge 2$ and $b>2$, then the other 

Let us assume that there are $p\ge 1$ copies of $K_4$ in $G_0$ and let $U$ be the set of their vertices. Recall that $G_2$ has at most 4 other vertices: $u,v$ connected to the  vertices in $A_1$ and $u',v'$ are connected to the vertices of $A_2$. Let us obtain $G_0'$ the following way. For each vertex $w\in U$, we delete the edges incident to $w$. If $w$ is connected to the vertices of $A_1$, we add $w$ to $A_2'$, and if $w$ is connected to the vertices of $A_2$, we add $w$ to $A_1'$. Then we add the remaining vertices of $U$ to $A_1'$ and $A_2'$ such that for each $K_4$, 2 of its vertices are in $A_1'$ and 2 are in $A'_2$. Finally, we add $A_1$ to $A_1'$ and $A_2$ to $A_2'$. Then we connect the vertices of $A_1'$ to the vertices of $A_2'$ and to $u$ and $v$, and similarly we connect the vertices of $A_2'$ to $u'$ and $v'$. Clearly $G_0'$ is $B_{3,1}$-free.

%... ha $w$ $A_1$-beli, akkor szomszédok csak nőnek, és azok közö szomszédai eddig is zömmel $A_1$-ben voltak... meg ugye $p/4$, az oké. Tehát $G_1$-belieknek nő a $H$-foka. Avagy: ami $H$ volt és pont $X$-ben metszette $G_1$-et, annak az $X$ része még mindig megvan. Tehát csak az lehet a gond ami nem metszette $G_1$-et? 

\begin{clm}
$G_0'$ contains more copies of $K_{a,b}$ than $G_0$.
\end{clm}

\begin{proof}
Assume first that $p\ge 4$. Then the degree of a vertex $w$ has not decreased. If $w\in A_1\cup A_2$, then the edges incident to $w$ in $G_0$ are also in $G_0'$. Otherwise,
$w$ was connected to at most $p+2$ vertices of $U$ in $G_0$ and is connected to $2p$ vertices of $U$ in $G_0'$. Furthermore, if $w$ was connected to vertices in $G_1$, those edges remain intact. Finally, $w$ was connected to at most 3 vertices out of $u,v,u',v'$. 

Consider a copy $H_0$ of $K_{a,b}$ in $G_0$ that intersects $G_1$ in a non-empty set $X$ of vertices. If $X$ intersects both partite sets of $G_1$, then $H_0$ is also in $G_0'$, as the edges of $G_0$ that are incident to vertices of $G_1$ are present in $G_0'$. If $X$ is a subset of $A_1$, then they belong to the same partite set of $H_0$ and the other partite set of $H_0$ belongs to the set of vertices connected to $A_1$ in $G_0$. Thus we have to pick the remaining vertices of $H_0$ from the common neighbors of these vertices; they have more common neighbors in $G_0'$, thus we can find more copies of $K_{a,b}$ this way in $G_0'$. This shows that there are more copies of $K_{a,b}$ in $G_0'$ intersecting $A_1\cup A_2$ than in $G_0$.

Finally, if a copy of $K_{a,b}$ does not intersect $G_1$, then we can pick it the following way. We pick a vertex $w$, then we pick either $b$ of its (at most $p+2$) neighbors in $G_2$ and $a-1$ of the at most $(p+2$) common neighbors of those in $G_2$, or we pick $a$ of its neighbors and $b-1$ of their common neighbors. We can pick more copies of $K_{a,b}$ in $G_0'$ the same way, as there are $2p$ neighbors of $w$ from $V(G_2)$ in $G_0'$ and those have at least $2p$ common neighbors from $V(G_2)$ in $G_0'$.

Assume now that $1\le p\le 3$. Observe that if $|A_1|=o(n)$ or $|A_2|=o(n)$ then we have $o(n^{a+b})$ copies of $K_{a,b}$ in $G_0$, less than in $K_{\lfloor n/2\rfloor,\lceil n/2\rceil}$, a contradiction. In a $K_4$, at most two vertices are connected to vertices in $G_1$, and the other at least two vertices have degree at most $4+p\le 7$ in $G_0$. By deleting the edges incident to them we removed $O(n^{b-1})$ copies of $K_{a,b}$. By adding them to $A_1$ or $A_2$, we added $\Omega(n^{a+b-1})$ copies of $K_{a,b}$, a contradiction.
\end{proof}

The above claim finishes the proof if $p\ge 1$.
Finally, if $p=0$, then $G_0$ without $uv$ and $u'v'$ is a complete bipartite graph, thus we are done if $m=1$ or if $u'v'$ does not exist. 
%If $|A_1|\le 2$ or $|A_2|\le 2$, then $G_0$ contains $O(n^b)$ copies of $H$ while $K_{\lfloor n/2\rfloor,\lceil n/2\rceil}$ contains $\Theta(n^{a+b})$ copies of $H$, a contradiction. Thus $|A_1|,|A_2|\ge 3$, which implies 
Observe that a triangle on the vertices $u,v,u',v'$ would create a $B_{3,1}$ with a vertex of $A_1$ or $A_2$. Thus, we can assume without loss of generality that $uu'$ and $vv'$ are not edges of $G_0$. Let $G_0''$ be the graph obtained by deleting from $G_0$ the edge $u'v'$ and adding $uu'$ and $vv'$. We have removed some copies of $K_{a,b}$ only if $a=b=2$ or if $a=1$. In the first case, we removed only one copy, and we created more copies of $K_{2,2}$. In the second case, the degree of every vertex remained the same or (for two vertices) increased. Thus in both cases, the number of copies of $K_{a,b}$ increased, a contradiction.
\end{proof}

\vskip 0.3truecm

\textbf{Funding}: Research supported by the National Research, Development and Innovation Office - NKFIH under the grants KH 130371, SNN 129364, FK 132060, and KKP-133819,  and by the Ministry of Education and Science of the Russian Federation in the framework of MegaGrant no. 075-15-2019-1926.

\vskip 0.3truecm

The authors declare that they have no known competing financial interests or personal relationships
that could have appeared to influence the work reported in this paper.

\end{document}